\documentclass[envcountsame]{llncs}
\usepackage{amsmath,amssymb}
\usepackage{stmaryrd}
\usepackage[hyperindex,breaklinks,pagebackref]{hyperref}

\usepackage{tikz}
\usetikzlibrary{fadings}

\def\Kol{\mathsf{K}}

\def\NON{\mathsf{NON}}

\def\ERT{\mathsf{ERT}}
\def\ECT{\mathsf{ECT}}
\def\I{\mathsf{I}}
\def\DIS{\mathsf{DIS}}

\def\r{\mathrm{r}}

\makeatletter
\newcommand{\douwidehat}[2]{%
  \sbox0{$\m@th#1\widehat{\hphantom{#2}}$}%
  \sbox2{$\m@th#1x$}
  \sbox4{$\m@th#1#2$}
  \dimen0=\ht0
  \advance\dimen0 -.8\ht2
  \dimen2=\dp4
  \rlap{%
    \raisebox{\dimexpr\dimen0-\dimen2}{%
      \scalebox{1}[-1]{\box0}%
    }%
  }%
  {#2}%
}
\makeatother

\usepackage{catchfile}

\usepackage{lineno}

\title{On the Complexity of\\ Computing G\"odel Numbers}

\author{Vasco Brattka\inst{1,2}\thanks{Vasco Brattka is supported by the National Research Foundation of South Africa.}}
             \institute{ Faculty of Computer Science, Universit\"at der Bundeswehr M\"unchen, Germany
             \and Dept.\ of Mathematics \& App.\ Maths., University of Cape Town, South Africa
           \email{Vasco.Brattka@cca-net.de}}

\begin{document} 

\maketitle



\def\AA{{\mathcal A}}
\def\BB{{\mathcal B}}
\def\CC{{\mathcal C}}
\def\DD{{\mathcal D}}
\def\EE{{\mathcal E}}
\def\FF{{\mathcal F}}
\def\GG{{\mathcal G}}
\def\HH{{\mathcal H}}
\def\II{{\mathcal I}}
\def\JJ{{\mathcal J}}
\def\KK{{\mathcal K}}
\def\LL{{\mathcal L}}
\def\MM{{\mathcal M}}
\def\NN{{\mathcal N}}
\def\OO{{\mathcal O}}
\def\PP{{\mathcal P}}
\def\QQ{{\mathcal Q}}
\def\RR{{\mathcal R}}
\def\SS{{\mathcal S}}
\def\TT{{\mathcal T}}
\def\UU{{\mathcal U}}
\def\VV{{\mathcal V}}
\def\WW{{\mathcal W}}
\def\XX{{\mathcal X}}
\def\YY{{\mathcal Y}}
\def\ZZ{{\mathcal Z}}


\def\bA{{\mathbf A}}
\def\bB{{\mathbf B}}
\def\bC{{\mathbf C}}
\def\bD{{\mathbf D}}
\def\bE{{\mathbf E}}
\def\bF{{\mathbf F}}
\def\bG{{\mathbf G}}
\def\bH{{\mathbf H}}
\def\bI{{\mathbf I}}
\def\bJ{{\mathbf J}}
\def\bK{{\mathbf K}}
\def\bL{{\mathbf L}}
\def\bM{{\mathbf M}}
\def\bN{{\mathbf N}}
\def\bO{{\mathbf O}}
\def\bP{{\mathbf P}}
\def\bQ{{\mathbf Q}}
\def\bR{{\mathbf R}}
\def\bS{{\mathbf S}}
\def\bT{{\mathbf T}}
\def\bU{{\mathbf U}}
\def\bV{{\mathbf V}}
\def\bW{{\mathbf W}}
\def\bX{{\mathbf X}}
\def\bY{{\mathbf Y}}
\def\bZ{{\mathbf Z}}


\def\IB{{\Bbb{B}}}
\def\IC{{\Bbb{C}}}
\def\IF{{\Bbb{F}}}
\def\IN{{\Bbb{N}}}
\def\IP{{\Bbb{P}}}
\def\IQ{{\Bbb{Q}}}
\def\IR{{\Bbb{R}}}
\def\IS{{\Bbb{S}}}
\def\IT{{\Bbb{T}}}
\def\IZ{{\Bbb{Z}}}

\def\IIB{{\Bbb{\mathbf B}}}
\def\IIC{{\Bbb{\mathbf C}}}
\def\IIN{{\Bbb{\mathbf N}}}
\def\IIQ{{\Bbb{\mathbf Q}}}
\def\IIR{{\Bbb{\mathbf R}}}
\def\IIZ{{\Bbb{\mathbf Z}}}


\def\ELSE{\quad\mbox{else}\quad}
\def\WITH{\quad\mbox{with}\quad}
\def\FOR{\quad\mbox{for}\quad}
\def\AND{\;\mbox{and}\;}
\def\OR{\;\mbox{or}\;}

\def\To{\longrightarrow}
\def\TO{\Longrightarrow}
\def\In{\subseteq}
\def\sm{\setminus}
\def\Inneq{\In_{\!\!\!\!/}}
\def\dmin{\mathop{\dot{-}}}
\def\splus{\oplus}
\def\SEQ{\triangle}
\def\DIV{\uparrow}
\def\INV{\leftrightarrow}
\def\SET{\Diamond}

\def\kto{\equiv\!\equiv\!>}
\def\kin{\subset\!\subset}
\def\pto{\leadsto}
\def\into{\hookrightarrow}
\def\onto{\to\!\!\!\!\!\to}
\def\prefix{\sqsubseteq}
\def\rel{\leftrightarrow}
\def\mto{\rightrightarrows}

\def\B{{\mathsf{{B}}}}
\def\D{{\mathsf{{D}}}}
\def\G{{\mathsf{{G}}}}
\def\E{{\mathsf{{E}}}}
\def\J{{\mathsf{{J}}}}
\def\K{{\mathsf{{K}}}}
\def\L{{\mathsf{{L}}}}
\def\R{{\mathsf{{R}}}}
\def\T{{\mathsf{{T}}}}
\def\U{{\mathsf{{U}}}}
\def\W{{\mathsf{{W}}}}
\def\Z{{\mathsf{{Z}}}}
\def\w{{\mathsf{{w}}}}
\def\HP{{\mathsf{{H}}}}
\def\C{{\mathsf{{C}}}}
\def\Tot{{\mathsf{{Tot}}}}
\def\Fin{{\mathsf{{Fin}}}}
\def\Cof{{\mathsf{{Cof}}}}
\def\Cor{{\mathsf{{Cor}}}}
\def\Equ{{\mathsf{{Equ}}}}
\def\Com{{\mathsf{{Com}}}}
\def\Inf{{\mathsf{{Inf}}}}

\def\Tr{{\mathrm{Tr}}}
\def\Sierp{{\mathrm Sierpi{\'n}ski}}
\def\psisierp{{\psi^{\mbox{\scriptsize\Sierp}}}}
\def\cl{{\mathrm{{cl}}}}
\def\Haus{{\mathrm{{H}}}}
\def\Ls{{\mathrm{{Ls}}}}
\def\Li{{\mathrm{{Li}}}}

\def\CL{\mathsf{CL}}
\def\ACC{\mathsf{ACC}}
\def\DNC{\mathsf{DNC}}
\def\ATR{\mathsf{ATR}}
\def\LPO{\mathsf{LPO}}
\def\LLPO{\mathsf{LLPO}}
\def\WKL{\mathsf{WKL}}
\def\RCA{\mathsf{RCA}}
\def\ACA{\mathsf{ACA}}
\def\SEP{\mathsf{SEP}}
\def\BCT{\mathsf{BCT}}
\def\IVT{\mathsf{IVT}}
\def\IMT{\mathsf{IMT}}
\def\OMT{\mathsf{OMT}}
\def\CGT{\mathsf{CGT}}
\def\UBT{\mathsf{UBT}}
\def\BWT{\mathsf{BWT}}
\def\HBT{\mathsf{HBT}}
\def\BFT{\mathsf{BFT}}
\def\FPT{\mathsf{FPT}}
\def\WAT{\mathsf{WAT}}
\def\LIN{\mathsf{LIN}}
\def\B{\mathsf{B}}
\def\BF{\mathsf{B_\mathsf{F}}}
\def\BI{\mathsf{B_\mathsf{I}}}
\def\C{\mathsf{C}}
\def\CF{\mathsf{C_\mathsf{F}}}
\def\CN{\mathsf{C_{\IN}}}
\def\CI{\mathsf{C_\mathsf{I}}}
\def\CK{\mathsf{C_\mathsf{K}}}
\def\CA{\mathsf{C_\mathsf{A}}}
\def\WPO{\mathsf{WPO}}
\def\WLPO{\mathsf{WLPO}}
\def\MP{\mathsf{MP}}
\def\BD{\mathsf{BD}}
\def\Fix{\mathsf{Fix}}
\def\Mod{\mathsf{Mod}}

\def\s{\mathrm{s}}
\def\r{\mathrm{r}}
\def\w{\mathsf{w}}

\def\leqm{\mathop{\leq_{\mathrm{m}}}}
\def\equivm{\mathop{\equiv_{\mathrm{m}}}}
\def\leqT{\mathop{\leq_{\mathrm{T}}}}
\def\lT{\mathop{<_{\mathrm{T}}}}
\def\nleqT{\mathop{\not\leq_{\mathrm{T}}}}
\def\equivT{\mathop{\equiv_{\mathrm{T}}}}
\def\nequivT{\mathop{\not\equiv_{\mathrm{T}}}}
\def\leqwtt{\mathop{\leq_{\mathrm{wtt}}}}
\def\equiPT{\mathop{\equiv_{\P\mathrm{T}}}}
\def\leqW{\mathop{\leq_{\mathrm{W}}}}
\def\equivW{\mathop{\equiv_{\mathrm{W}}}}
\def\leqtW{\mathop{\leq_{\mathrm{tW}}}}
\def\leqSW{\mathop{\leq_{\mathrm{sW}}}}
\def\equivSW{\mathop{\equiv_{\mathrm{sW}}}}
\def\leqPW{\mathop{\leq_{\widehat{\mathrm{W}}}}}
\def\equivPW{\mathop{\equiv_{\widehat{\mathrm{W}}}}}
\def\leqFPW{\mathop{\leq_{\mathrm{W}^*}}}
\def\equivFPW{\mathop{\equiv_{\mathrm{W}^*}}}
\def\leqWW{\mathop{\leq_{\overline{\mathrm{W}}}}}
\def\nleqW{\mathop{\not\leq_{\mathrm{W}}}}
\def\nleqSW{\mathop{\not\leq_{\mathrm{sW}}}}
\def\lW{\mathop{<_{\mathrm{W}}}}
\def\lSW{\mathop{<_{\mathrm{sW}}}}
\def\nW{\mathop{|_{\mathrm{W}}}}
\def\nSW{\mathop{|_{\mathrm{sW}}}}
\def\leqt{\mathop{\leq_{\mathrm{t}}}}
\def\equivt{\mathop{\equiv_{\mathrm{t}}}}
\def\leqtop{\mathop{\leq_\mathrm{t}}}
\def\equivtop{\mathop{\equiv_\mathrm{t}}}

\def\bigtimes{\mathop{\mathsf{X}}}

\def\leqm{\mathop{\leq_{\mathrm{m}}}}
\def\equivm{\mathop{\equiv_{\mathrm{m}}}}
\def\leqT{\mathop{\leq_{\mathrm{T}}}}
\def\leqM{\mathop{\leq_{\mathrm{M}}}}
\def\equivT{\mathop{\equiv_{\mathrm{T}}}}
\def\equiPT{\mathop{\equiv_{\P\mathrm{T}}}}
\def\leqW{\mathop{\leq_{\mathrm{W}}}}
\def\equivW{\mathop{\equiv_{\mathrm{W}}}}
\def\nequivW{\mathop{\not\equiv_{\mathrm{W}}}}
\def\leqSW{\mathop{\leq_{\mathrm{sW}}}}
\def\equivSW{\mathop{\equiv_{\mathrm{sW}}}}
\def\leqPW{\mathop{\leq_{\widehat{\mathrm{W}}}}}
\def\equivPW{\mathop{\equiv_{\widehat{\mathrm{W}}}}}
\def\nleqW{\mathop{\not\leq_{\mathrm{W}}}}
\def\nleqSW{\mathop{\not\leq_{\mathrm{sW}}}}
\def\lW{\mathop{<_{\mathrm{W}}}}
\def\lSW{\mathop{<_{\mathrm{sW}}}}
\def\nW{\mathop{|_{\mathrm{W}}}}
\def\nSW{\mathop{|_{\mathrm{sW}}}}

\def\botW{\mathbf{0}}
\def\midW{\mathbf{1}}
\def\topW{\mathbf{\infty}}

\def\pol{{\leq_{\mathrm{pol}}}}
\def\rem{{\mathop{\mathrm{rm}}}}

\def\cc{{\mathrm{c}}}
\def\d{{\,\mathrm{d}}}
\def\e{{\mathrm{e}}}
\def\ii{{\mathrm{i}}}

\def\Cf{C\!f}
\def\id{{\mathrm{id}}}
\def\pr{{\mathrm{pr}}}
\def\inj{{\mathrm{inj}}}
\def\cf{{\mathrm{cf}}}
\def\dom{{\mathrm{dom}}}
\def\range{{\mathrm{range}}}
\def\graph{{\mathrm{graph}}}
\def\Graph{{\mathrm{Graph}}}
\def\epi{{\mathrm{epi}}}
\def\hypo{{\mathrm{hypo}}}
\def\Lim{{\mathrm{Lim}}}
\def\diam{{\mathrm{diam}}}
\def\dist{{\mathrm{dist}}}
\def\supp{{\mathrm{supp}}}
\def\union{{\mathrm{union}}}
\def\fiber{{\mathrm{fiber}}}
\def\ev{{\mathrm{ev}}}
\def\mod{{\mathrm{mod}}}
\def\sat{{\mathrm{sat}}}
\def\seq{{\mathrm{seq}}}
\def\lev{{\mathrm{lev}}}
\def\mind{{\mathrm{mind}}}
\def\arccot{{\mathrm{arccot}}}

\def\Add{{\mathrm{Add}}}
\def\Mul{{\mathrm{Mul}}}
\def\SMul{{\mathrm{SMul}}}
\def\Neg{{\mathrm{Neg}}}
\def\Inv{{\mathrm{Inv}}}
\def\Ord{{\mathrm{Ord}}}
\def\Sqrt{{\mathrm{Sqrt}}}
\def\Re{{\mathrm{Re}}}
\def\Im{{\mathrm{Im}}}
\def\Sup{{\mathrm{Sup}}}

\def\LSC{{\mathcal LSC}}
\def\USC{{\mathcal USC}}

\def\CE{{\mathcal{E}}}
\def\Pref{{\mathrm{Pref}}}

\def\Baire{\IN^\IN}

\def\TRUE{{\mathrm{TRUE}}}
\def\FALSE{{\mathrm{FALSE}}}

\def\co{{\mathrm{co}}}

\def\BBB{{\tt B}}

\newcommand{\SO}[1]{{{\mathbf\Sigma}^0_{#1}}}
\newcommand{\SI}[1]{{{\mathbf\Sigma}^1_{#1}}}
\newcommand{\PO}[1]{{{\mathbf\Pi}^0_{#1}}}
\newcommand{\PI}[1]{{{\mathbf\Pi}^1_{#1}}}
\newcommand{\DO}[1]{{{\mathbf\Delta}^0_{#1}}}
\newcommand{\DI}[1]{{{\mathbf\Delta}^1_{#1}}}
\newcommand{\sO}[1]{{\Sigma^0_{#1}}}
\newcommand{\sI}[1]{{\Sigma^1_{#1}}}
\newcommand{\pO}[1]{{\Pi^0_{#1}}}
\newcommand{\pI}[1]{{\Pi^1_{#1}}}
\newcommand{\dO}[1]{{{\Delta}^0_{#1}}}
\newcommand{\dI}[1]{{{\Delta}^1_{#1}}}
\newcommand{\sP}[1]{{\Sigma^\P_{#1}}}
\newcommand{\pP}[1]{{\Pi^\P_{#1}}}
\newcommand{\dP}[1]{{{\Delta}^\P_{#1}}}
\newcommand{\sE}[1]{{\Sigma^{-1}_{#1}}}
\newcommand{\pE}[1]{{\Pi^{-1}_{#1}}}
\newcommand{\dE}[1]{{\Delta^{-1}_{#1}}}

\newcommand{\dBar}[1]{{\overline{\overline{#1}}}}

\def\QED{$\hspace*{\fill}\Box$}
\def\rand#1{\marginpar{\rule[-#1 mm]{1mm}{#1mm}}}

\def\BL{\BB}


\newcommand{\bra}[1]{\langle#1|}
\newcommand{\ket}[1]{|#1\rangle}
\newcommand{\braket}[2]{\langle#1|#2\rangle}

\newcommand{\ind}[1]{{\em #1}\index{#1}}
\newcommand{\mathbox}[1]{\[\fbox{\rule[-4mm]{0cm}{1cm}$\quad#1$\quad}\]}


\newenvironment{eqcase}{\left\{\begin{array}{lcl}}{\end{array}\right.}

\newtheorem{fact}[theorem]{Fact}
\newtheorem{slogan}[theorem]{Slogan}
\newtheorem{thesis}[theorem]{Thesis}


\begin{abstract}
Given a computable sequence of natural numbers, it is a natural task to find a G\"odel number of a program
that generates this sequence. It is easy to see that this problem is neither continuous nor computable.
In algorithmic learning theory this problem is well studied from several perspectives and one question
studied there is for which sequences this problem is at least learnable in the limit.
Here we study the problem on all computable sequences and we classify the Weihrauch 
complexity of it. For this purpose we can, among other methods, utilize the amalgamation technique known from learning theory.
As a benchmark for the classification we use closed and compact choice problems and their jumps on natural numbers,
and we argue that these problems correspond to induction and boundedness principles, as they are known from the Kirby-Paris hierarchy
in reverse mathematics.
We provide a topological as well as a computability-theoretic classification, which reveal some significant differences. 
\end{abstract}


\section{Introduction}

Given a sequence of natural numbers such as
\[1,2,3,...\]
it is a well-known game to guess how the sequence continuous.
While the first guess could be that it is the sequence of all positive natural numbers,
one has to reconsider this guess when the continuation
\[1,2,3,5,7,...\]
appears. A new guess might be that it is the sequence of all odd numbers together with $2$,
but when we see
\[1,2,3,5,7,11,13,17,19,23,...,\]
then it rather looks like the sequence of prime numbers together with $1$.
Even if the given sequence is computable, it is easy to see that that there is
neither a continuous nor a computable way to make the guess converge in general.
Questions like this have been extensively studied in algorithmic learning theory~\cite{OSW90,ZZ08}.
More recently, AI approaches to determine programs for given sequences have been tested~\cite{GOU23}.
In algorithmic learning theory
it is known and easy to see, for instance, that if one restricts oneself to primitive-recursive
sequences, then there is an algorithm that makes the guess converge eventually. 
The crucial problem here is whether it is recognizable that a guess (in form of a G\"odel number
generating the sequence) represents a total sequence, which is always the case 
if one restricts everything a priorily to G\"odel codes of primitive-recursive sequences.
However for general G\"odel numbers totality is a $\mathrm{\Pi^0_2}$--question in the arithmetical hierarchy 
and hence an oracle
such as $\emptyset''$ is required.
In fact Gold proved the following result~\cite{Gol67}.

\begin{theorem}[Gold 1967]
The total computable sequences are not learnable in the limit.
\end{theorem}

That means that there is no general algorithm that could produce a sequence of G\"odel numbers
that converges to a correct G\"odel number from a total computable sequence that is given as input.

We formalize this problem as follows
The {\em G\"odel problem} $\G:\In\IN^\IN\mto\IN$ that we want to study can be defined
by 
\[\G:\In\IN^\IN\mto\IN,p\mapsto\{i\in\IN:\varphi_i=p\},\] 
where
the domain $\dom(\G)$ consists of all computable sequences and 
$\varphi:\IN\to\PP$ denotes some standard G\"odel numbering of the set of computable sequences $\PP$.
Briefly we could also define $\G$ as $\varphi^{-1}$ restricted to total computable sequences.
If we want to consider $\G$ as a theorem whose complexity we aim to classify, then
the theorem would be the statement that every total computable function has a G\"odel number, i.e.,
\[(\forall\mbox{ computable }p\in\IN^\IN)(\exists i\in\IN)\;\varphi_i=p.\]
That is, this theorem states that $\varphi$ is surjective and hence a numbering of $\PP$.
This theorem seems to be very simple and does not appear to need particularly powerful 
resources. However, we will see that it shows some peculiar properties.

Another perspective one could take is to ask what additional useful information is carried by a program $i$ that the sequence $p$
itself does not make accessible? This question has been studied by Hoyrup and Rojas~\cite{HR17}
and their answer could be summarized briefly as follows:

\begin{slogan}[Hoyrup and Rojas 2017]
\label{slogan:Hoyrup-Rojas}
{\em
The only useful additional information carried by a program compared to the natural number
sequence it represents, is an upper bound on the Kolmogorov complexity of the sequence.}
\end{slogan}

One of our goals is to study in which sense this slogan can be converted into theorems
on the Weihrauch complexity of the G\"odel problem. For this purpose we introduce a
number of further related problems.
For our study the {\em Kolmogorv complexity} is the function
$\Kol:\In\IN^\IN\mto\IN$ with $\dom(\Kol)=\dom(\G)$ defined by
\[\Kol:\In\IN^\IN\to\IN,p\mapsto\min\G(p)\]
for all $p\in\dom(\Kol)$. The study by Hoyrup and Rojas also motivates to investigate
the following variant of $\G$ 
\[\G_\geq:\In\IN^\IN\times\IN\mto\IN,(p,m)\mapsto\G(p)\]
with $\dom(\G_\geq):=\{(p,m)\in\IN^\IN\times\IN:m\geq \Kol(p)\}$. That is
$\G_\geq$ is the G\"odel problem that gets an upper bound on the Kolmogorov complexity 
as additional input information.
We note that the output does not need to satisfy the input bound $m$ according to this definition.
Yet another function that one can consider is
\[\Kol_\geq:\IN^\IN\mto\IN,p\mapsto\{m\in\IN:m\geq \Kol(p)\}\]
that just yields an upper bound on the Kolmogorov complexity of the input sequence,
again with $\dom(\Kol_\geq):=\dom(\G)$.

In terms of their Weihrauch complexity these problems are in the following obvious relation
that is also visualized in the diagram in Figure~\ref{fig:basic-problems}.

\begin{lemma}
\label{lem:basic-problems}
$\G_\geq\sqcup\Kol_\geq\leqSW\G\leqSW\Kol$ and $\G=\G_\geq\circ(\id,\K_\geq)$.
\end{lemma}

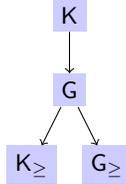
\begin{figure}[htb]
\begin{center}
\begin{tikzpicture}[scale=1,auto=left]
\node[style={fill=blue!20}]  (GK) at (-3.5,3) {$\G_\geq$};
\node[style={fill=blue!20}]  (KG) at (-4.5,3) {$\Kol_\geq$};
\node[style={fill=blue!20}]  (G) at (-4,4) {$\G$};
\node[style={fill=blue!20}]  (K) at (-4,5) {$\Kol$};
\draw[->] (K) edge (G);
\draw[->] (G) edge (KG);
\draw[->] (G) edge (GK);
\end{tikzpicture}
\caption{The G\"odel problem and its relatives.}
\label{fig:basic-problems}
\end{center}
\end{figure}

Obviously, the Kolmogorov complexity $\Kol$ computes $\G$ and $\G$ computes $\G_\geq$ as well as $\K_\geq$.
Here $\sqcup$ stands for the supremum with respect to Weihrauch reducibility~\cite{BGP21}.
A precise definition of Weihrauch reducibility follows in Section~\ref{sec:weihrauch}.

Our goal is to classify the Weihrauch complexity and also its topological counterpart
of the problems mentioned above. 
The way we will calibrate this complexity is with the help of the problems
\[\K_\IN\lW\C_\IN\lW\K_\IN'\lW\C_\IN'\lW\K_\IN''\lW\C_\IN''\lW...\]
The so-called {\em compact choice problem} $\K_\IN$ and the so-called {\em closed choice problem}
$\C_\IN$ on the natural numbers play an important role in Weihrauch complexity~\cite{BGP21}.
As noted by the author and Rakotoniaina~\cite[\S7]{BR17} they can be seen as Weihrauch complexity analogs of the Kirby-Paris hierarchy 
\[\B\mathrm{\Sigma}^0_1\leftarrow\I\mathrm{\Sigma}^0_1\leftarrow\B\mathrm{\Sigma}^0_2\leftarrow\I\mathrm{\Sigma}^0_2\leftarrow\B\mathrm{\Sigma}^0_3\leftarrow\I\mathrm{\Sigma}^0_3\leftarrow...\]
of boundedness and induction principles 
as it is known from reverse mathematics~\cite{Sim09,Hir15,DM22,HP93}
and we will simply refer to this hierarchy as the (Weihrauch version of the) {\em Kirby-Paris hierarchy} in the following.
They are introduced in Section~\ref{sec:choice}.
Using the results presented there we will argue in the Conclusions (see Section~\ref{sec:conclusions})
that they yield natural counterparts of the induction and boundedness problems
known from reverse mathematics~\cite{Sim09}.
In Section~\ref{sec:topological} we will classify the topological Weihrauch degree of the
G\"odel problem and its variant. We will also study the question, which oracle 
among $\emptyset',\emptyset'',...$ is optimal to validate our classification.
In Section~\ref{sec:computability} we classify the computability-theoretic Weihrauch degree
of these problems, which turns out to be significantly different.
In Section~\ref{sec:closure-lower} we discuss closure properties of $\G$ that help us
to say something on lower bounds on these Weihrauch degrees.

\section{Weihrauch complexity}
\label{sec:weihrauch}

In this section we introduce some basic definitions of Weihrauch complexity.
A more detailed survey can be found in \cite{BGP21}.
We recall that a {\em represented space} $(X,\delta)$ is a set $X$ together with a 
{\em representation} $\delta:\In\IN^\IN\to X$, i.e., a surjective potentially partial map $\delta$.
If $f:\In X\mto Y$ is a multivalued partial map on represented spaces $(X,\delta_X)$ and $(Y,\delta_Y)$,
then $F:\In\IN^\IN\to\IN$ is called a {\em realizer} of $f$, if $\delta_YF(p)\in f\delta_X(p)$ for all $p\in\dom(f\delta_X)$.
We denote the fact that $F$ realizes $f$ by $F\vdash f$.
We consider multivalued maps with realizers as problems.

\begin{definition}[Problem]
A multivalued map $f:\In X\mto Y$ on represented spaces with a realizer is called a {\em problem}.
\end{definition}

Now we are prepared to define Weihrauch complexity and some variants of it.
By $\id:\IN^\IN\to\IN^\IN$ we denote the identity on Baire space $\IN^\IN$,
by $\langle.\rangle$ we denote a standard pairing function on $\IN^\IN$.
We also use the angle bracket for countable tupling functions, Cantor tupling functions on $\IN$
and corresponding pairing functions for $\IN\times\IN^\IN$.

\begin{definition}[Weihrauch complexity]
\label{def:Weihrauch}
Let $f:\In X\mto Y$ and $g:\In Z\mto W$ be two multi-valued functions. 
\begin{enumerate}
\item $f$ is {\em Weihrauch reducible} to $g$, in symbols $f\leqW g$, if there are computable $H,K:\In\IN^\IN\to\IN^\IN$ such that $H\langle\id,GK\rangle\vdash f$ whenever $G\vdash g$.
\item $f$ is {\em strongly Weihrauch reducible} to $g$, in symbols $f\leqSW g$, if there are computable $H,K:\In\IN^\IN\to\IN^\IN$ such that
         $HGK\vdash f$ whenever $G\vdash g$.
\end{enumerate}
We write $\leq_\mathrm{W}^p$ and $\leq_\mathrm{sW}^p$ for the relativized versions of this reducibilities, where
$H,K$ are only required to be computable relative to $p\in\IN^\IN$. Analogously, we write $\leq_\mathrm{W}^*$ and $\leq_\mathrm{sW}^*$
if $H,K$ are only required to be continuous.
\end{definition}

The diagram in Figure~\ref{fig:Weihrauch} illustrates the situation of $\leqW$ and its relativizations.

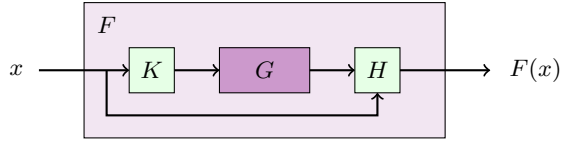
\begin{figure}[htb]
\begin{center}
\begin{tikzpicture}[scale=.3,auto=left]
\draw[style={fill=violet!10}] (-4,6) rectangle (12,0);
\draw[style={fill=green!10}]  (8,4) rectangle (10,2);
\draw[style={fill=green!10}]  (-2,4) node (v1) {} rectangle (0,2);
\draw[style={fill=violet!40}]  (2,4) rectangle (6,2);
 \node at (-1,3) {$K$};
\node at (9,3) {$H$};
\node at (4,3) {$G$};
\node at (-3,5) {$F$};
\node at (-7,3) {$x$};
\node at (16,3) {$F(x)$};
\draw[->,thick] (0,3) -- (2,3);
\draw[->,thick] (-6,3) -- (-2,3);
\draw[->,thick] (6,3) -- (8,3) ;
\draw[->,thick] (10,3) -- (14,3);
\draw[->,thick] (-3,3) -- (-3,1) -- (9,1) -- (9,2);
\end{tikzpicture}
\end{center}
\caption{Weihrauch reducibility.}
\label{fig:Weihrauch}
\end{figure}

We use the usual notations $\equivW$ and $\equivSW$ for the corresponding equivalences,
and similarly also for the relativized versions.
The distributive lattice induced by $\leqW$ is usually referred to as {\em Weihrauch lattice}.

By $f^*,\widehat{f}$ and $f^\diamond$ we denote the {\em finite parallelization}, {\em parallelization},
and {\em diamond} operations on problems. 
The definitions can be found in \cite{BGP21}, except for the diamond operation
that was introduced by Neumann and Pauly~\cite{NP18}.
While $f^*$ can be seen as closure under the parallel product $f\times f$,
$f^\diamond$ can be seen as closure under the compositional product $f\star f$ by a result of Westrick~\cite{Wes21} (see below).
Again, the definitions of $\times$ and $\star$ can be found in \cite{BGP21}.
For our purposes we just mention the characterization that
\[f\star g\equivW\max\nolimits_{\leqW}\{f_0\circ g_0:f_0\leqW f\mbox{ and }g_0\leqW g\},\]
which was proved in \cite[Corollary~3.7]{BP18}.
Likewise, we obtain by a theorem of Westrick \cite[Theorem~1]{Wes21} that every problem $f$ with $\id\leqW f$ satisfies
\[f^\diamond\equivW\min\nolimits_{\leqW}\{g:f\leqW g\star g\leqW g\},\]
i.e., $f^\diamond$ gives us the smallest Weihrauch degree above $f$ that is closed under
compositional product. We will also need the {\em first-order version} of a problem $f:\In X\mto Y$
that can be characterized according to~\cite[Theorem~2.2]{DSY23} by
\[^1f\equivW\max\nolimits_{\leqW}\{g:\mbox{$g:\In\IN^\IN\mto\IN$ and $g\leqW f$}\}.\]
The first-order part $^1f$ captures the strongest problem with codomain $\IN$ that is below $f$.
This concept was recently introduced by
Dzhafarov, Solomon, and Yokoyama~\cite{DSY23} and has also been studied by Valenti~\cite{Val21}
and Valenti and Sold\`a~\cite{SV22}.

Finally, we also need the {\em jump} $f'$ of a problem that was introduced in \cite{BGM12}.
For a problem $f:\In X\mto Y$, the {\em jump} $f':\In X'\mto Y$
is set-theoretically speaking the same problem as $f$, but the representation
of the input set $X$ is changed. If $\delta$ is the representation of $X$,
then $\delta':=\delta\circ\lim$ is the representation of $X'$.
The $n$--th jump of $f$ is denoted by $f^{(n)}$.
More details can be found in \cite{BGP21}.

We close this section with the definition of a number of problems that we are going to use in the following.
By $\Tr\In\{0,1\}^*$ we denote the set of binary trees and by $[T]$ the set of infinite paths of $T\in\Tr$.
By $\widehat{n}=nnn...\in\IN^\IN$ we denote the constant sequence with value $n\in\IN$.

The problems $\LPO$ and $\LLPO$ have been called {\em limited principle of omniscience} and
{\em lesser limited principle of omniscience}, respectively, by Bishop~\cite{Bis67} in the context of constructive mathematics.

\begin{definition}[Some problems]
We define the following problems.
\begin{enumerate}
\item $\LPO:\IN^\IN\to\{0,1\},\LPO(p)=1:\iff p=\widehat{0}$
\item $\LLPO:\In\IN^\IN\mto\{0,1\},i\in\LLPO\langle p_0,p_1\rangle:\iff p_i\not=\widehat{0}$,\\
with $\dom(\LLPO)=\{\langle p_0,p_1\rangle\in\IN^\IN:\langle p_0,p_1\rangle\not=\widehat{0}\}$. 
\item $\lim_\IN:\IN^\IN\to\IN,p\mapsto\lim_{n\to\infty}p(n)$\hfill (limit on natural numbers)
\item $\B:\In\IN^\IN\mto\IN,p\mapsto\{m\in\IN:(\forall n\in\IN)\;p(n)\leq m\}$\hfill (boundedness problem)
\item $\inf:\In\IN^\IN\to\IN,p\mapsto\min\C_\IN(p)$ \hfill (least number problem)
\item $\min:\IN^\IN\to\IN,p\mapsto\min\range(p)$ \hfill (minimum problem)
\item $\CL_\IN:\In\IN^\IN\mto\IN,p\mapsto\{n\in\IN:n\mbox{ cluster point of }p\}$ \hfill (cluster point problem)
\item $\BWT_\IN$ is the restriction of $\CL_\IN$ to bounded sequences \hfill (Bolzano-Weierstra\ss{})
\item $\lim\inf\nolimits_\IN:\In\IN^\IN\to\IN,p\mapsto\min\CL_\IN(p)$ \hfill (limit inferior problem)
\item $\WKL:\In\Tr\mto2^\IN,T\mapsto[T]$ \hfill (Weak K\H{o}nig's lemma)
\end{enumerate}
\end{definition}

All definitions are meant with their natural domains, i.e., $\lim_\IN$ is defined on converging sequences,
$\B$ on bounded sequences, $\CL_\IN$ of sequences with cluster points, etc.

\section{Closed and compact choice on the natural numbers}
\label{sec:choice}

Now we want to introduce our benchmark problems, which are {\em closed choice} on natural numbers
and {\em compact choice} on natural numbers.
Choice problems have been studied for a while and they have been uniformly defined for more
general spaces than the natural numbers (see~\cite{BG11a,BBP12,BGM12,BGP21} for more information).
For our purposes it is sufficient to define these choice problems on $\IN$.

\begin{definition}[Closed and compact choice]
We define {\em closed choice} $\C_\IN$ and {\em compact choice} $\K_\IN$ on natural numbers as follows:
\begin{enumerate}
\item $\C_\IN:\In\IN^\IN\mto\IN, p\mapsto\{n\in\IN:(\forall k)\;p(k)\not=n\}$,\\
        with $\dom(\C_\IN)=\{p\in\IN^\IN:\range(p)\subsetneqq\IN\}$,
\item $\K_\IN:\In\IN^\IN\times\IN\mto\IN,(p,m)\mapsto\{n\leq m:(\forall k)\;p(k)\not=n\}$,\\
        with $\dom(\K_\IN)=\{(p,m)\in\IN^\IN\times\IN:\range(p)\subsetneqq\{0,...,m\}\}$.
\end{enumerate}
\end{definition}

Hence, the task of $\C_\IN$ is to find a natural number which is not enumerated by the input sequence.
Compact choice $\K_\IN$ has essentially the same task, except that only numbers $n\leq m$ below
some additionally given bound $m$ are considered.

The following was proved by Valenti~\cite[Page~98, Proposition~4.48, Corollary~4.50]{Val21} and Sold\`a and Valenti~\cite[Proposition~7.1, Theorem~7.2, Corollary~7.6]{SV22}.
An independent proof of the first equivalences in each item is due to Dzhafarov, Solomon, and Yokoyama~\cite[Theorems~4.1 and 4.2]{DSY23}.
We did not find an explicit statement of the second equivalence in (2), hence this is probably new and we add a proof,
which can easily be derived using the methods of Sold\'a and Valenti.
We note that $\C_2^*\equivSW\LLPO^*\equivSW\K_\IN$ by \cite[Proposition~10.9]{BGM12}.

\begin{theorem}[Closed and compact choice]
\label{thm:first-order}
For all $n\in\IN$:
\begin{enumerate}
\item $^1(\lim\nolimits^{(n)})\equivSW\C_\IN^{(n)}\equivW(\LPO^{(n)})^\diamond$,
\item $^1(\WKL^{(n)})\equivSW\K_\IN^{(n)}\equivW(\LLPO^{(n)})^\diamond$.
\end{enumerate}
\end{theorem}
\begin{proof}
It only remains to proof the second equivalence in the second item.
It follows from~\cite[Theorem~4.40, Corollary~4.44]{Val21}
and Theorem~\ref{thm:first-order} as they imply
\[\K_\IN^{(n)}\equivW\!\,^1(\WKL^{(n)})\equivW\!\,^1(\widehat{\LLPO^{(n)}})\equivW(\LLPO^{(n)})^\diamond,\]
since $\LLPO^{(n)}$ is complete by~\cite[Propositions~4.19, 6.3]{BG21a},
$\widehat{\LLPO}\equivSW\WKL$ by \cite[Theorem~8.2]{BG11} and parallelization commutes
with jumps by \cite[Proposition~5.7~(3)]{BGM12}.
This implies $(\K_\IN^{(n)})^\diamond\equivW\K_\IN^{(n)}$.
\qed
\end{proof}

It is worth formulating the special case for $n=1$ of Theorem~\ref{thm:first-order}.

\begin{corollary}[Closed and compact choice]
\label{cor:closed-compact}
We obtain:
\begin{enumerate}
\item $\LPO^\diamond\equivW\C_\IN\equivW\!\,^1 \lim$,
\item $\LLPO^\diamond\equivW\K_\IN\equivW\!\,^1 \WKL$.
\end{enumerate}
\end{corollary}

Neumann and Pauly~\cite[Proposition~10]{NP18} first proved the statement for $\LPO$
and this statement as well as the one for $\LLPO$ is also included in work by
Valenti and Sold\`a~\cite[Proposition~7.1, Theorem~7.2]{SV22}.
We directly obtain the following closure properties under composition.

\begin{corollary}[Closure under composition]
\label{cor:closure-composition}
For all $n\in\IN$ we obtain:
\begin{enumerate}
\item $\C_\IN^{(n)}\star\C_\IN^{(n)}\equivW\C_\IN^{(n)}$,
\item $\K_\IN^{(n)}\star\K_\IN^{(n)}\equivW\K_\IN^{(n)}$.
\end{enumerate}
\end{corollary}

This was known for $\C_\IN$ by unpublished work of the author, H\"olzl and Kuyper (2017)
and by  \cite[Theorem~7.2]{SV22},
but seemingly not for $\K_\IN$.\footnote{In fact, Corollary~\ref{cor:closure-composition} shows that 
\cite[Proposition~8.13]{BG21a} is incorrect as $\K_\IN$ is indeed not complete.}

In \cite[Proposition~7.2]{BR17} it was proved that closed and compact choice
and their jumps are linearly ordered in the following way.

\begin{fact}
$\K_\IN^{(n)}\lSW\C_\IN^{(n)}\lSW\K_\IN^{(n+1)}$ for all $n\in\IN$.
\end{fact}

Altogether, these properties of closed and compact choice justify to use them and their jumps as benchmark 
problems for the classification of the complexity of G\"odel problem $\G$ and its variants, which are all first-order
problems.

We close this section with some further characterizations of some of the choice problems.
Most of these are well-known, except perhaps the one on the limes inferior.

{\begin{proposition}[Least number problem, minimum, limes inferior]
\label{prop:lim-inf}
\begin{enumerate}
\item $\K_\IN\equivSW\LLPO^*\lSW\min\equivSW\LPO^*\lSW\C_\IN\equivSW\lim_\IN\equivSW\inf\equivW\B$,
\item $\K_\IN'\equivSW\BWT_\IN\lSW\min'\lSW\C_\IN'\equivSW\CL_\IN\equivSW\lim\inf\equivSW\inf'$.
\end{enumerate}
\end{proposition}}
\begin{proof}
$\K_\IN\equivSW\LLPO^*$ was proved in \cite[Proposition~10.9]{BGM12}.
$\LPO^*\equivSW\min$ was stated in \cite[Proposition~11.7.22]{BGP21}.
We obtain $\LPO^*\nleqW\LLPO^*$, since we have
$\LLPO^*\leqW\WKL$,
but $\LPO\nleqW\WKL$, as $\widehat{\LPO}\equivSW\lim\nleqW\WKL$.
The strict reduction $\LPO^*\lSW\C_\IN$ has been proved in \cite[Proposition~3.5, Example~3.12]{BBP12}.
For $\C_\IN\equivSW\lim_\IN\equivSW\inf$ see \cite[Theorem~11.7.13]{BGP21}.
Here the equivalence $\C_\IN\equivSW\inf$ is originally from \cite[Proposition~7.1]{BR17}.
The equivalence $\C_\IN\equivW\B$ is easy to see and was essentially proved in \cite[Proposition~3.3]{BG11a}.
The equivalence $\K_\IN'\equivSW\BWT_\IN$ and $\C_\IN'\equivSW\CL_\IN$ are from \cite[Corollaries~9.10 and 11.10]{BGM12}
and \cite{BCG+17}, 
the separations in  $\K_\IN'\lSW\min'\lSW\C_\IN'$ follow from \cite[Theorem~11]{BHK17}, since
the separations without jump are already topological separations.
The equivalence $\C_\IN'\equivSW\inf'$ follows as jumps are monotone with respect to $\equivSW$.
It finally remains to prove $\C_\IN'\equivSW\lim\inf$. This follows from the previous equivalence,
by \cite[Proposition~3.6]{BHK18}, which tells us that the jump of the input space in $\inf$ (that represents
closed sets by enumeration of the complement), is equivalent to representing the closed set
as set of cluster points of a sequence.
Instead of using $\C_\IN\equivSW\inf$ one can also use $\LPO^*\equivSW\min$ to conclude
\[\lim\inf\nolimits_\IN\leqSW\lim\nolimits_\IN\circ\widehat{\min}\leqSW\lim\nolimits_\IN\circ\lim\equivSW\lim\nolimits_\IN'\equivSW\C_\IN'.\]
Here the first reduction follows as a point $p\in\IN^\IN$ can be translated into a sequence $(p_n)_{\in\IN}$ in $\IN^\IN$
such that $p_n$ contains exactly those numbers that appear at least $n$--times in $p$. Then $\min(p_n)$ converges to
the least cluster point.
The lower bound follows from $\C_\IN'\equivSW\CL_\IN\leqSW\lim\inf_\IN$.
\qed
\end{proof}

We emphasize the following fact.

\begin{fact}
\label{fact:boundedness}
$\C_\IN\equivW\B\lSW\C_\IN$. 
\end{fact}

The separation follows from \cite[Proposition~11.6.18]{BGP21}.

\section{The topological classification}
\label{sec:topological}

The purpose of this section is to classify $\G,\K,\G_\geq$ and $\K_\geq$ from a topological perspective.
This as such is pretty simple and we obtain the following result.

\begin{theorem}[The topological degree]
\label{thm:topological}
$\B\equiv_{\rm sW}^*\Kol_\geq\equiv_{\rm W}^*\G\equiv_{\rm sW}^*\K\equiv_{\rm sW}^*\C_\IN$ and $\G_\geq$ is continuous.
\end{theorem}

The statement is illustrated in the diagram in Figure~\ref{fig:Godel-topological} using our benchmark problems $\K_\IN$ and $\C_\IN$.
We omit the simple proof, as we soon prove a stronger result.

The equivalences $\Kol_\geq\equiv_{\rm W}^*\G$ and $\G_\geq\equiv_{\rm W}^*\id$ 
can be seen as formal versions of the slogan of Hoyrup and Rojas.

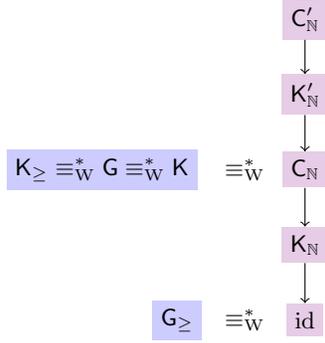
\begin{figure}[htb]
\begin{center}
\begin{tikzpicture}[scale=1,auto=left]
\node[style={fill=violet!20}]  (CNP) at (0,6) {$\C_\IN'$};
\node[style={fill=violet!20}]  (KNP) at (0,5) {$\K_\IN'$};
\node[style={fill=violet!20}]  (CN) at (0,4) {$\C_\IN$};
\node[style={fill=violet!20}]  (KN) at (0,3) {$\K_\IN$};
\node[style={fill=violet!20}]  (ID) at (0,2) {$\id$};
\node[style={fill=blue!20}]  (GK) at (-1.7,2) {$\G_\geq$};
\node[style={fill=blue!20}]  (KG) at (-2.7,4) {$\Kol_\geq\equiv_{\rm W}^*\G\equiv_{\rm W}^*\K$};
\node  (E1) at (-0.8,4) {$\equiv_{\rm W}^*$};
\node  (E2) at (-0.8,2) {$\equiv_{\rm W}^*$};
\draw[->] (CNP) edge (KNP);
\draw[->] (KNP) edge (CN);
\draw[->] (CN) edge (KN);
\draw[->] (KN) edge (ID);
\end{tikzpicture}
\caption{The G\"odel problem in the topological Weihrauch lattice.}
\label{fig:Godel-topological}
\end{center}
\end{figure}

As continuity is computability with respect to some oracle, the question appears whether among
the jumps $\emptyset^{(n)}$ there is a simplest one in the place of $*$
that makes the statements of Theorem~\ref{thm:topological} correct.
We recall that the set $\mathsf{TOT}:=\{i\in\IN:\varphi_i$ total$\}$ is a $\mathrm{\Pi^0_2}$--complete set in the arithmetical
hierarchy~\cite[Proposition~X.9.6]{Odi99} and hence $\mathsf{TOT}\leqT\emptyset''$. 
Hence it follows easily that $\emptyset''$ is sufficient in place of $*$, as totality of the functions
represented by G\"odel numbers is a useful property that can be utilized to provide a simple proof of
Theorem~\ref{thm:topological}.

Somewhat surprisingly and with a little more effort, we can show that actually the halting problem $\emptyset'$ in the place of $*$
also validates Theorem~\ref{thm:topological}.
We start with proving the following equivalences.
The upper bound is easy to obtain and for the lower bound the essential idea is to use (a variant) of the set of random natural numbers.

\begin{theorem}[With the halting problem as oracle]
\label{thm:halting}
$\B\equiv_{\rm sW}^{\emptyset'}\Kol_\geq\equiv_{\rm W}^{\emptyset'}\G\equiv_{\rm sW}^{\emptyset'}\K\equiv_{\rm sW}^{\emptyset'}\C_\IN$.
\end{theorem}
\begin{proof}
We use the fact that $\C_\IN\equivSW\lim_\IN\equivW\B$ according to Proposition~\ref{prop:lim-inf}.
By Lemma~\ref{lem:basic-problems} we have $\Kol_\geq\leqSW\G\leqSW\K$.

We can prove $\Kol\leq_{\mathrm{sW}}^{\emptyset'}\lim_\IN$ as follows.
Given a computable sequence $p\in\IN^\IN$,
we go through all G\"odel numbers $i=0,1,2,...$ one by one. In case of $i$  we check for each $n=0,1,2,...$ whether
$n\in\dom(\varphi_i)$ (with the help of the halting problem) and whether $\varphi_i(n)=p(n)$. If so, then we write $i$ to the output $q$ and we move on to the next $n$.
If one of these tests fails, then we move on to the next $i$. This procedure stops going to the next $i$ when the smallest $i$ with
$\varphi_i=p$ is reached. From this point on only $i$ will be written to the output, i.e., $q$ is eventually constant with value $i$.
Hence $\lim_\IN(q)=i$, as desired. 
The same construction also proves $\K_\geq\leq_{\mathrm{sW}}^{\emptyset'}\B$.

For the other direction we use the following slight generalization of the set of {\em random natural numbers}:
\[R:=\{\langle k,n\rangle\in\IN:\min\{i\in\IN:\varphi_i(k)=n\}\geq n\}.\]
For every fixed $k\in\IN$ there are infinitely many $n\in\IN$ such that $\langle k,n\rangle\in R$. 
For $k=0$ this is Kolmogorov's result that there are infinitely many random natural numbers~\cite[Proposition~III.2.12]{Odi89}.
For each other fixed $k$ this can be proved analogously.
Moreover, $R$ is easily seen to be co-c.e.\ and hence $R\leqT\emptyset'$. 

We prove $\B\leq_\mathrm{sW}^{\emptyset'}\Kol_\geq$.
Given a bounded sequence $q\in\IN^\IN$ we can assume that this is monotone increasing.
The goal is to find a computable $p\in\IN^\IN$ such that every $b\in\Kol_\geq(p)$ is an upper bound of $q$.
Hence, with the help of the halting problem we start to search some $n\geq q(0)$ such that $\langle 0,n\rangle\in R$. 
We then set $p(0):=n$ and we continue inductively.
Given $p(0),...,p(k)$ for some $k\in\IN$, we check whether $q(k+1)>q(k)$.
If so, then we search for some $n\geq q(k+1)$ with $\langle k+1,n\rangle\in R$ and we set $p(k+1)=n$.
Otherwise, we choose $p(k+1)=p(k)$. Since the sequence $q$ is bounded and monotone, it is eventually constant
and hence $p$ is also eventually constant and thus computable.
By construction we obtain that every $b\in\Kol_\geq(p)$ is an upper bound of $q$.

For the proof of $\lim_\IN\leq_\mathrm{sW}^{\emptyset'}\G$ we refine this idea somewhat.
We are given a converging sequence $q\in\IN^\IN$.
The goal is now to find a computable $p\in\IN^\IN$ such that $\varphi_i(i)=\lim q$ for every $i\in\G(p)$.
We start with searching an $n>0$ with $\langle 0,n\rangle\in R$ and we set
$p(0):=n$ and $p(1):=q(0)$. 
We continue inductively.
Given $p(0),...,p(2k+1)$ for some $k\in\IN$, we check whether $q(k+1)\not=q(k)$.
If so, then we search for some $n>2k+2$  with $\langle 2k+2,n\rangle\in R$  
and we set $p(2k+2):=n$ and $p(2k+3):=q(k+1)$.
Otherwise, we choose $p(2k+2):=p(2k+3):=q(k+1)$. 
Since $q$ is convergent, the sequence $p$ converges too and has the same limit as $q$.
By construction, every $i\in G(p)$ is large enough so that $p(i)=\lim p=\lim q$ 
and hence we obtain $\varphi_i(i)=p(i)=\lim q$.
\qed
\end{proof}

A close inspection of the proof reveals that the oracle $p$ has only been used
on the input side, i.e., for the computation of the preprocessing operation $K$ in terms
of Definition~\ref{def:Weihrauch}. 
We denote such a kind of Weihrauch reduction here temporarily by $\leq_\mathrm{sW}^{0,p}$.
In \cite[Theorem~11]{BHK17} we proved $f'\leq_\mathrm{W}^p g'\TO f\leq_\mathrm{W}^{p'}g$.
Using the oracle $p$ only on the input side allows for the following more symmetric result,
which is essentially built upon the limit control technique that was developed in \cite{BHK17a,Bra18}.

\begin{proposition}[Jumps and oracles]
\label{prop:jumps-oracles}
For all problems $f,g$ and $p\in\IN^\IN$ we have:
$f\leq_\mathrm{sW}^{0,p'}g\iff f'\leq_\mathrm{sW}^{0,p}g'$.
\end{proposition}
\begin{proof}
Without loss of generality, we can assume that $f,g$ are of type $f,g:\In\IN^\IN\mto\IN^\IN$
and hence $f'\equivSW f\circ\lim$ and $g'\equivSW g\circ\lim$.

Let $f\leq_{\mathrm{sW}}^{0,p'}g$ hold via
functions $H,K:\In\IN^\IN\to\IN$ such that $H$ is computable and $K$ is computable
relative to $p'$. 
Then $H\circ g\circ K\prefix f$. 
Since $K$ is computable relative to $p'$, the function $K\circ\lim$ is limit computable relative to $p$
and hence there is a function $K_0:\In\IN^\IN\to\IN^\IN$ that is computable relative to $p$ by \cite[Theorem~2.18]{Bra18}
such that $K\circ\lim=\lim\circ K_0$. 
We obtain
\[H\circ g\circ\lim\circ K_0=H\circ g\circ K\circ\lim\prefix f\circ\lim,\]
which implies $f'\equivSW f\circ\lim\leq_\mathrm{sW}^{0,p} g\circ\lim\equivSW g'$.

Let now $f\circ\lim\equivSW f'\leq_\mathrm{sW}^{0,p}g'\equivSW g\circ\lim$ hold via
functions $H,K:\In\IN^\IN\to\IN$ such that $H$ is computable and $K$ is computable
relative to $p$. 
Then $H\circ g\circ\lim\circ K\prefix f\circ\lim$.
By \cite[Theorem~3.7]{Bra18} there is a $K_0:\In\IN^\IN\to\IN^\IN$ that is computable relative to $p'$ such
that $K_0\prefix\lim\circ K\circ\lim^{-1}$ and hence, as $\lim$ is surjective,
\[H\circ g\circ K_0\prefix H\circ g\circ\lim\circ K\circ\lim\nolimits^{-1}\prefix f\]
follows, which implies $f\leq_\mathrm{sW}^{0,p'}g$.
\qed
\end{proof}

With Proposition~\ref{prop:jumps-oracles} we obtain the following corollary of Theorem~\ref{thm:halting}.

\begin{corollary}[Jumps]
\label{cor:jumps}
$\B^{(n)}\equivSW\Kol_\geq^{(n)}$ and $\G^{(n)}\equivSW\K^{(n)}\equivSW\C_\IN^{(n)}$ for $n\geq1$.
\end{corollary}

The idea of the next proof is inspired by the 
{\em amalgamation technique} that is attributed by~\cite{ZZ08} to Wiehagen~\cite{Wie78}.

\begin{proposition}
\label{prop:Ggeq-halting}
$\G_\geq$ is computable relative to the halting problem.
\end{proposition}
\begin{proof}
We want to prove that given an input $\langle p,m\rangle\in\dom(\G_\geq)$ we can
compute a G\"odel code $i$ with $\varphi_i=p$ with the help of the halting problem $\emptyset'$.

We consider the following relation $\approx$ defined on the set $\PP$ of partial computable
functions $f,g:\In\IN\to\IN$:
\[f\approx g:\iff(\forall n\in\dom(f)\cap\dom(g))\;f(n)=g(n).\]
In this case, we say that $f$ and $g$ are {\em compatible}.
We emphasize that $\approx$ is reflexive and symmetric, but not transitive and hence $\approx$ is not
an equivalence relation.\footnote{This can be seen, as for instance the nowhere defined function is compatible
with any other partial computable function, but not all partial computable functions are compatible with each other.}
The set $C:=\{\langle i,j\rangle\in\IN:\varphi_i\approx\varphi_j\}$ of codes of compatible functions
is easily seen to be co-c.e.\ and hence $C\leqT\emptyset'$.
We now consider the following sets for $i\leq m$ that we call {\em pockets}:
\[P_i:=\{j\leq m:\varphi_i\approx\varphi_j\}.\]
That is, we collect all codes $j\leq m$ of functions compatible to $\varphi_i$ in the pocket $P_i$.
With the help of $C$, we can compute all these pockets as finite sets.
We call a pocket $P_i$ {\em compatible}, if $\varphi_{j_0}\approx\varphi_{j_1}$ holds for all $j_0,j_1\in P_i$. 
We can decide with the help of $C$ whether a pocket $P_i$ is compatible.
Among the pockets $P_0,...,P_m$ we remove all incompatible pockets and all double occurrences of
the same pocket. This yields a list $P_{i_0},...,P_{i_k}$ of pairwise different pockets,
which are all compatible by themselves.

We now claim that no pocket in this list is a subset of another pocket in the list.
For a contradiction we assume without loss of generality that $P_{i_0}\In P_{i_1}$.
Since the pockets are not identical, this means $P_{i_0}\subsetneqq P_{i_1}$.
Then there is some $j\in P_{i_1}\setminus P_{i_0}$. Since $i_0\in P_{i_0}\In P_{i_1}$
and $P_{i_1}$ is compatible, we have that $\varphi_{i_0}\approx\varphi_j$, which implies $j\in P_{i_0}$.
This is a contradiction and proves the claim.

We now claim that among the pockets $P_{i_0},...,P_{i_k}$
\begin{enumerate}
\item exactly one contains
at least one code $j$ with $\varphi_j=p$ and all other codes $j$ in this pocket satisfy $\varphi_j\approx p$,
\item all other pockets contain at least one $j$ with $\varphi_j\not\approx p$.
\end{enumerate}
Since there is some $i\leq m$ with $\varphi_i=p$ by assumption, we can consider the
pocket $P_i$. Since $\varphi_i$ is total, the pocket $P_i$ is compatible and hence it clearly satisfies (1).
Now let $P_{i'}$ be a pocket among $P_{i_0},...,P_{i_k}$ with $P_{i'}\not=P_i$.
Since there are no double occurrences among the pockets, this is the same as saying
that $P_{i'}$ is any other pocket with $i'\in\{i_0,...,i_k\}\setminus\{i\}$. Suppose that all $j\in P_{i'}$ satisfy $\varphi_j\approx p$.
Then $p=\varphi_i$ implies $j\in P_i$ and hence $P_{i'}\In P_i$, which is impossible by the previous claim.
This proves (2).

Let us now say that a pocket $P_i$ is {\em compatible with $p$}, if $p\approx\varphi_j$ for all $j\in P_i$.
Conditions (1) and (2) guarantee that there is exactly one pocket among the $P_{i_0},...,P_{i_k}$ that
is compatible with $p$ and this pocket contains also a G\"odel number of $p$.
We can now search for a prefix of $p$ that suffices to identify all pockets among $P_{i_0},...,P_{i_k}$ that
are incompatible with $p$. For this purpose we just need to find an $n\in\IN$ and $j$ in the corresponding pocket
with $n\in\dom(\varphi_j)$ and $\varphi_j(n)\not=p(n)$. After a finite search we have identified all pockets 
that are incompatible with $p$ and there
remains only one pocket $P_i$ that is then necessarily compatible with $p$. 

From the index $i\in\{i_0,...,i_k\}$ of this pocket we can compute a G\"odel number $r(i)$ of $p$,
i.e., such that $\varphi_{r(i)}=p$ as follows: for each input $n\in\IN$ we search for some 
$j\in P_i$ such that $n\in\dom(\varphi_j)$ and we produce $\varphi_j(n)$ as result.
Such a value $j$ must exist, as $P_i$ contains G\"odel numbers of $p$ and for any
$j\in P_i$ with $n\in\dom(\varphi_j)$ we obtain $\varphi_j(n)=p(n)$ as $P_i$ is compatible with $p$.

Hence, $r(i)\in\G_\geq\langle p,m\rangle$. 
\qed
\end{proof}

We note that $r(i)\leq m$ is not required and might not hold.

\section{The computability-theoretic classification}
\label{sec:computability}

We now want to explore the unrelativized Weihrauch degree of the G\"odel problem and its variants.
It turns out that the situation is significantly different in this case and none of $\G,\Kol$ and $\Kol_\geq$ is
computably Weihrauch equivalent to $\C_\IN$. 
Since no other natural candidates of problems are known that are topologically but not computably
equivalent to $\C_\IN$, the best that we can expect for a classification is to obtain optimal upper bounds 
in terms of the Weihrauch version of the Kirby-Paris hierarchy
\[\K_\IN\lW\C_\IN\lW\K_\IN'\lW\C_\IN'\lW\K_\IN''\lW\C_\IN''\lW...\]
that are
as narrow as possible. We start with some positive results. For the upper bound of $\G_\geq$
we use again the amalgamation technique.

\begin{proposition}[Upper bounds]
\label{prop:upper}
We obtain 
\begin{enumerate}
\item $\Kol_\geq\leqSW\G\leqSW\Kol\leqSW\Kol'\equivSW\C_\IN'$,
\item $\G_\geq\leqSW\LPO^*\lW\C_\IN$.
\end{enumerate}
\end{proposition}
\begin{proof}
(1) follows from Lemma~\ref{lem:basic-problems} and Corollary~\ref{cor:jumps}.
$\Kol\leqSW\C_\IN'$ could also easily be proved directly using $\C_\IN'\equivSW\lim\inf_\IN$, which holds
according to Proposition~\ref{prop:lim-inf}: one just needs to enumerate
all G\"odel numbers $i$ repeatedly, whenever they validate longer and longer portions of the input. Those numbers that
appear infinitely often are the correct ones.

(2) We use $\LPO^*\equivSW\min$ and we prove $\G_\geq\leqSW\min$.
We are given an input $\langle p,k\rangle$ with a computable $p\in\IN^\IN$ and
$\Kol(p)\leq k$. The goal is to find a G\"odel code $j$ of $p=\varphi_j$.
We use the encoding $d:\In 2^\IN\to\IN$ with $d(A):=\sum_{i\in A}2^{-i}$ for finite sets $A\In\IN$. 
The idea is to determine a shrinking sequence $(A_j)_{j\in\IN}$
of such sets that converges to a set $A$.
As output we produce a sequence $(d(A_j))_{j\in\IN}$ of codes for the
sets $A_j$. We start with $A_0:=\{0,...,k\}$
and we remove $i\leq k$ from $A_j$ whenever we find an $m\in\IN$ such that $\varphi_i(m)$ converges
but $\varphi_i(m)\not=p(m)$. In this situation $A_{j+1}:=A_j\setminus\{i\}$. 
Whenever we do not find a suitable $m$ for some time, then we set $A_{j+1}:=A_j$.
In both cases we write $d(A_j)$ to the output.
It is clear that the limit $A:=\lim_{j\to\infty}A_j$ exists and $d(A)=\min_{j\in\IN}(d(A_j))$.
The set  $A$ consists of all G\"odel codes up to $k$  that belong to potentially partial restrictions of $p$.
Among those, there is also a G\"odel code of the total function $p$ itself, as guaranteed by the bound $k$. 
From $d(A)$ we can compute a G\"odel code $j$ of $p$. The G\"odel code $j$ belongs to the machine
that works as follows: upon input $n\in\IN$, we find the first
among the values in $\varphi_i(n)$ for $i\in A$ that exists and we write the corresponding
value on the output.
\qed
\end{proof}

The result $\G_\geq\leqW\C_\IN$ is essentially contained in~\cite{FW79} and has also been used in \cite{HR17}. 
Our improvement is that we bring the upper bound down to $\LPO^*$.
The diagram in Figure~\ref{fig:Godel} illustrates the situation.

\begin{figure}[htb]
\begin{center}
\begin{tikzpicture}[scale=1,auto=left]
\node[style={fill=violet!20}]  (CNP) at (0,6) {$\G'\equivSW\Kol'\equivSW\C_\IN'$};
\node[style={fill=violet!20}]  (KNP) at (0,5) {$\K_\IN'$};
\node[style={fill=violet!20}]  (CN) at (0,4) {$\C_\IN$};
\node[style={fill=violet!20}]  (LPOS) at (0,3) {$\LPO^*$};
\node[style={fill=violet!20}]  (KN) at (0,2) {$\K_\IN$};
\node[style={fill=violet!20}]  (ID) at (0,0) {$\id$};
\node[style={fill=violet!20}]  (DIS) at (0,1) {$\DIS$};
\node[style={fill=violet!20}]  (LPO) at (-1,2) {$\LPO$};

\node[style={fill=blue!20}]  (GK) at (-2.5,2) {$\G_\geq$};
\node[style={fill=blue!20}]  (KG) at (-4,2) {$\Kol_\geq$};
\node[style={fill=blue!20}]  (G) at (-4,4) {$\G$};
\node[style={fill=blue!20}]  (K) at (-2.5,5) {$\Kol$};

\draw[->] (CNP) edge (KNP);
\draw[->] (KNP) edge (CN);
\draw[->] (CN) edge (LPOS);
\draw[->] (LPOS) edge (KN);
\draw[->] (LPOS) edge (LPO);
\draw[->] (KN) edge (DIS);
\draw[->] (LPO) edge (DIS);
\draw[->] (DIS) edge (ID);

\draw[->] (CNP) edge (K);
\draw[->] (K) edge (G);
\draw[->] (G) edge (KG);
\draw[->] (G) edge (GK);
\draw[->] (LPOS) edge (GK);
\draw[->] (KG) edge (ID);
\draw[->] (GK) edge (ID);
\draw[->] (K) edge (LPO);
\end{tikzpicture}
\caption{The G\"odel problem in the Weihrauch lattice.}
\label{fig:Godel}
\end{center}
\end{figure}
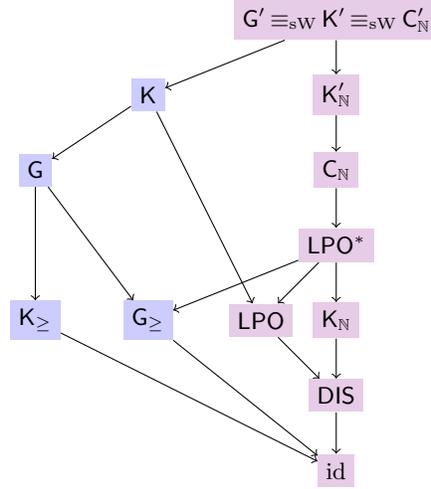

Next we want to show that the upper bounds given in Proposition~\ref{prop:upper} are minimal,
at least with respect to our benchmark problems $\K_\IN\lSW\C_\IN\lSW\K_\IN'\lSW\C_\IN'$
of the Kirby-Paris hierarchy.
We use finite extension constructions for this purpose.

\begin{proposition}[Optimality of upper bounds]
\label{prop:optimal-upper}
We obtain 
\begin{enumerate}
\item $\Kol_\geq\nleqW\K_\IN'$,
\item $\G_\geq\nleqW\K_\IN$.
\end{enumerate}
\end{proposition}
\begin{proof}
(1)
Suppose $\Kol_\geq\leqW\K_\IN'$ holds. 
This implies $\Kol_\geq\leqW\BWT_\IN$ by Proposition~\ref{prop:lim-inf}.
Hence there are computable $H,K:\In\IN^\IN\to\IN^\IN$ such that
$H\langle\id,G K\rangle$ is a realizer for $\Kol_\geq$ whenever $G$ is a realizer for $\BWT_\IN$.
We start with a constant zero input $p_0=000...$ and we use machines for $H,K$ 
in order to construct a computable $p\in\IN^\IN$ such that $H\langle p,GK(p)\rangle$ is not an upper bound for a G\"odel number of $p$.
For a contradiction, let us assume the contrary.
Let $k_n$ be the smallest number that is enumerated a maximal number of times in $K(p_0)|_n$.
If $n$ is large enough, then $k_n$ appears infinitely often in $K(p_0)$. 
Let $m_0$ be a computable joint modulus of continuity of $K$ and $H$ in the sense that a prefix of length $m_0(n)$ of $p_0$ suffices to produce an output
$K(p_0)$ of length $n$ and an also an output $j:=H\langle p_0,k_n\rangle$ provided that $k_n$ appears infinitely often in $K(p_0)$.
We can assume that $m$ is strictly monotone increasing. 
Then one of the numbers $i\leq j$ must be a G\"odel number of $p_0$.
Hence, we can search for some pair $n_0,i_0\in\IN$ with $n_0>0$ such that $j_0:=H\langle p_0,k_{n_0}\rangle$
and $\varphi_{i_0}(m_0(n_0))$ exist with $i_0\leq j_0$. If we have found such a pair $n_0,i_0$, then we set
$p_1:=p_0(0)....p_0(m_0(n_0)-1)(1\dmin\varphi_{i_0}(m_0(n_0)))000...$.
We now continue this algorithm inductively with $p_1$ and we analogously search for some $n_1>n_0$ and $i_1\leq j_1$
as before and we set 
$p_2:=p_1(0)....p_1(m_1(n_1)-1)(1\dmin\varphi_{i_1}(m_1(n_1)))000...$.
Here we assume that $m_1\geq m_0$ pointwise.
We can even restrict our search to $i_1$ with $i_1\not=i_0$, as $i_0$ cannot be a G\"odel code of $p_1$ by construction.
This algorithm determines a sequence $(p_k)$ that converges to some computable $p$,
as the algorithm shows how to compute $p$ and we can assume that the corresponding $i_k$ are pairwise different.
Since $p$ is computable, the sequence $K(p)$ must be bounded and there is some length $n$ that suffices that any number
that appears a maximal number of times in $K(p)|_k=K(p_k)|_k$ for $k\geq n$
actually appears infinitely often in $K(p)$. 
Hence, all the $j_k=H\langle p_k,k_{n_k}\rangle=H\langle p,k_{n_k}\rangle$ with $k\geq n$ are upper bounds of G\"odel numbers of $p$.
However, since there are only finitely many different $k_{n_k}$ with $k\geq n$ there can only
be finitely many $j_k$, which is a contradiction to the fact that there are infinitely many pairwise different $i_k\leq j_k$ with $k\geq n$.

(2) Suppose $\G_\geq\leqW\K_\IN$.
Then there are computable functions $H,K:\In\IN^\IN\to\IN^\IN$
such that $H\langle\id,GK\rangle$ is a realizer of $\G_\geq$ whenever $G$ is a realizer of $\K_\IN$.
We construct a computable $p\in\IN^\IN$ together with some bound $k\in\IN$ on its
Kolmogorov complexity such that the reduction does not work for input $\langle p,k\rangle$.
Let $p_0:=\widehat{0}$ and $k_0$ such that $\Kol(p_0)\leq k_0$.
Then $K\langle p_0,k_0\rangle$ is a name of a finite set
$A_0:=\{a_1,...,a_m\}$, given by negative information together with a bound on the largest
element of $A_0$. Then $i_j:=H\langle\langle p_0,k_0\rangle,a_j\rangle$ for $j=1,...,m$
have to be G\"odel numbers of $p$.
We evaluate $K\langle p_0,k_0\rangle$ until we have a candidate $B_0=\{a_1,...,a_m\}$ for $A_0$ (that satisfies $A_0\In B_0$)
and then we evaluate $i_j=H\langle\langle p_0,k_0\rangle,a_j\rangle$ for $j=1,...,m$ under
the assumption that our candidate is correct. 
By continuity of $K,H$ some finite prefix $p|_{n_0}$ of length $n_0\in\IN$ is sufficient to provide some sufficiently good candidate $B_0$ for $A_0$ 
such that we can compute all the corresponding numbers $i_1,...,i_m$.
Now we simultaneously try to evaluate $\varphi_{i_j}(n_0+j)$ for $j=1,...,m$, which might fail for some $i_j$ as the candidate $B_0$ for $A_0$ is still larger than the true $A_0$. 
Hence, simultaneously we continue to find a better candidate $B_0$ for $A_0$ by evaluation of $K\langle p_0,k_0\rangle$.
If, at some stage, we find a smaller candidate $B_0$ for $A_0$ than previously after inspecting a prefix of $p$ longer 
than $n_0$, then we restart the evaluation of
the $\varphi_{i_j}(n_0+j)$ for $j=1,...,m$ with the revised numbers $n_0,m$. 
At some stage our candidate $B_0$ for $A_0$ is by continuity of $K$ small enough such all these values $i_1,...,i_m$ and $\varphi_{i_j}(n_0+j)$ are available and then we set
\[p_1:=0^{n_0+1}(1\dmin\varphi_{i_1}(n_0+1))...(1\dmin\varphi_{i_m}(n_0+m))\widehat{0}\]
and we compute a G\"odel number $k_1$ of $p_1$ that satisfies $k_1>k_0$.
By $w$ we denote the prefix of $p_1$ of length $n_0+1+m$.
Clearly, all the numbers $i_1,...,i_m$ cannot be G\"odel codes for $p_1$. 
We now consider the set $A_1=\{a_1,...,a_m\}$ named by $K(p_1,k_1)$ and
we inductively continue this construction with $\langle p_1,k_1\rangle$ in place of
$\langle p_0,k_0\rangle$, where we ensure that the corresponding number $n_1\in\IN$,
i.e., the length up to which we have to evaluate $p_1$, satisfies $n_1>n_0+1+m$.
In this case we set 
\[p_2:=w0^{n_1-n_0-m}(1\dmin\varphi_{i_1}(n_1+1))...(1\dmin\varphi_{i_m}(n_1+m))\widehat{0}.\]
In this way we obtain a sequence $(p_j)_{j\in\IN}$ that converges to a computable
$p\in\IN^\IN$, since our algorithm shows how to compute the limit. We also obtain an
unbounded sequence $(k_j)_{j\in\IN}$ of G\"odel numbers and a corresponding 
sequence $(A_j)_{j\in\IN}$ of finite sets named by $K\langle p_j,k_j\rangle$, as well as a sequence of finite sets $(B_j)_{j\in\IN}$ with our final candidates $B_j\supseteq A_j$ for the $A_j$.
One of these numbers $k:=k_j$ has to satisfy $\Kol(p)\leq k$ and we now
consider the input $\langle p,k\rangle$. Then $K\langle p,k\rangle$ is a name
for a finite set $A\In B_j$. But all the finitely many numbers
$i\in B:=H\langle\langle p_j,k_j\rangle,B_j\rangle$ are by construction
not G\"odel codes of $p_{j+1}$ and likewise the subset
$H\langle\langle p,k\rangle,A\rangle\In B$ does not contain any G\"odel codes of
$p$. This means that the reduction fails on the computable input $\langle p,k\rangle$
with $\Kol(p)\leq k$.
\qed
\end{proof}

As $^1\WKL^{(n)}\equivW\K_\IN^{(n)}$ and $^1\lim^{(n)}\equivW\C_\IN^{(n)}$ by Theorem~\ref{thm:first-order},
we obtain the following corollary
with respect to $\WKL$ and $\lim$.

\begin{corollary}[Weak K\H{o}nig's lemma and limits]
\label{cor:WKL-lim}
We obtain
\begin{enumerate}
\item $\Kol_\geq\leqSW\G\leqSW\Kol\leqSW\lim'$, but $\Kol_\geq\nleqW\WKL'$,
\item  $\G_\geq\leqSW\lim$, but $\G_\geq\nleqW\WKL$.
\end{enumerate}
\end{corollary}

Even more roughly, we can formulate a conclusion regarding the Borel complexity of the G\"odel problem problem and its relatives. We recall that by results of \cite[Section~9]{Bra05}, we have the following.
\pagebreak 

\begin{fact}[Borel measurability]\nopagebreak
\label{fact:Borel}
We obtain for every problem $n\in\IN$
\begin{enumerate}
\item $f\leqW\lim^{(n)}\iff f$ is effectively $\mathbf{\Sigma^0_{n+2}}$--measurable,
\item $f\leqSW\lim^{(n)}\iff f$ is $\mathbf{\Sigma^0_{n+2}}$--measurable.
\end{enumerate}
\end{fact}

This leads to the following conclusions.

\begin{corollary}[Borel complexity]
\label{cor:Borel}
We obtain
\begin{enumerate}
\item $\G, \Kol$ and $\Kol_\geq$ are all effectively $\mathbf{\Sigma^0_{3}}$--Borel measurable, and topologically $\mathbf{\Sigma^0_2}$--Borel measurable, but not effectively so.
\item $\G_\geq$ is effectively $\mathbf{\Sigma^0_{2}}$--Borel measurable and continuous, but not computable.
\end{enumerate}
\end{corollary}

In fact, even finer conclusions regarding measurability properties of the G\"odel problem and its relatives 
(compare \cite[Theorem~11.9.3]{BGP21}). But we leave these to the reader.

Finally, we still want to separate $\Kol_\geq$ from $\G_\geq$ and $\G$. By $|_\mathrm{W}$ we denote incomparability with respect to $\leqW$.

\begin{proposition}
\label{prop:separation}
$\G_\geq|_\mathrm{W}\Kol_\geq$ and $\Kol_\geq\lW\G$.
\end{proposition}
\begin{proof}
$\Kol_\geq\nleqW\G_\geq$ follows since $\G_\geq$ is continuous and $\Kol_\geq$ is not.
Now let us assume for a contradiction that $\G_\geq\leqW\Kol_\geq$.
Then there are computable $H,K:\In\IN^\IN\to\IN^\IN$ such that $H\langle\id,GK\rangle$ is a 
realizer for $\G_\geq$ whenever $G$ is a realizer for $\Kol_\geq$.
Since $K$ is computable, there is a total computable $r:\IN\to\IN$ such that
$K(\varphi_i)(n)=\varphi_{r(i)}(n)$ for all $i,n\in\IN$. More precisely, one should
use a machine for $K$ here for which a prefix $K(\varphi_i)(0),...,K(\varphi_i)(n)$ might be defined,
even if $\varphi_i\not\in\dom(K)$. Given an input $(p,m)\in\dom(\G_\geq)$ it is clear
that $j:=\max\{r(0),...,r(m)\}$ is an upper bound for a G\"odel code for $K(p)$ since $m$ is an upper bound for a G\"odel code of $p$.
Hence $H\langle p,j\rangle$ is a G\"odel code for $p$. But this means $\G_\geq\leqSW\id$, as $H$ is computable. 
This contradicts Proposition~\ref{prop:optimal-upper}.

As $\G_\geq\leqW\G$ and $\Kol_\geq\leqW\G$, it follows that $\G\leqW\Kol_\geq$ would imply $\G_\geq\leqW\Kol_\geq$,
which we just disproved. Hence $\Kol_\geq\lW\G$.
\qed
\end{proof}

Hence, the diagram in Figure~\ref{fig:Godel} is complete in the sense that no further Weihrauch reductions to the 
problems from our Kirby-Paris hierarchy are possible, besides the shown ones (and those that follow by reflexivity and transitivity).

Proposition~\ref{prop:separation} also shows that the slogan of Hoyrup and Rojas is not true in a 
computability-theoretic sense. 
In fact, in this sense G\"odel numbers carry strictly more information about the sequences that
they represent than just an upper bound of their Kolmogorov complexity.

\section{Closure properties and lower bounds}
\label{sec:closure-lower}

So far we have provided upper bounds to the G\"odel problem together with a proof that
these upper bounds are optimal with respect to with respect to the Weihrauch version of the Kirby-Paris hierarchy.
Now we want to provide some closure properties of $\G$, some of which we will utilize to provide optimal
lower bounds in a certain sense. By $f|_\mathrm{c}$ we denote the restriction of a problem $f$
to all computable instances.

\begin{theorem}[Closure properties]
\label{thm:closure}
We obtain the following:
\begin{enumerate}
\item $\widehat{\G}|_{\rm c}\equivSW\G<_\mathrm{W}^*\widehat{\G}$,
\item $(\G\star\G)|_\mathrm{c}\equivSW\G$,
\item $\G^*\equivSW\G$.
\end{enumerate}
\end{theorem}
\begin{proof}
(1) Given a computable sequence $(p_n)_{n\in\IN}$ in $\IN^\IN$, we can compute
$p:=\langle p_0,p_1,p_2,...\rangle$ and from any $i\in\G(p)$, i.e., such that $\varphi_i=p$, we can compute
some $r(i,n)$ such that $\varphi_{r(i,n)}=p_n$. Hence, $\widehat{\G}|_{\rm c}\equivSW\G$.
We also obtain with Theorem~\ref{thm:topological}
\[\G\equiv_\mathrm{W}^*\C_\IN<_\mathrm{W}^*\lim\equiv_\mathrm{W}^*\widehat{\C_\IN}\equiv_\mathrm{W}^*\widehat{\G}.\]

(2) It is clear that $\G\leqSW(\G\star\G)|_\mathrm{c}$.
For the opposite reduction we use the cylindrification theorem~\cite[Lemma~3.10]{BP18} that guarantees that there
is a computable $F:\In\IN^\IN\to\IN^\IN$ such that
\[\G\star\G\equivSW\langle\id\times\G\rangle\circ F\circ\langle\id\times\G\rangle.\]
We are going to prove that $(\langle\id\times\G\rangle\circ F\circ\langle\id\times \G\rangle)|_\mathrm{c}\leqSW\widehat{\G}|_{\mathrm c}$,
which is sufficient by (1). 
We use the projection functions $\pi_i:\IN^\IN\to\IN^\IN,\langle p_1,p_2\rangle\mapsto p_i$ for $i\in\{1,2\}$.
One difficulty now is that for G\"odel numbers $i\in\IN$ of non-total functions $\varphi_i$
the value $F\langle p,i\rangle$ does neither need to exist nor does it need to be computable. 
However, without loss of generality, we can assume that $\langle p,i\rangle\in\dom(F)$ if and only if $i$ is the index of a total function.
We just evaluate $\varphi_i(n)$ for $n=0,1,2,...$ internally in the algorithm for $F$ and we only
let $F\langle p,i\rangle(n)$ be defined if the former value exists.

The fact that $F\langle p,i\rangle$ is undefined for some $i$, 
is the reason why we work with a modification $F_0$ of $F$
that we define as follows: given $p,i$ we try to evaluate $F\langle p,i\rangle(n)$ for $n=0,1,2,...$. 
If the value $F\langle p,i\rangle(k)$ is found for all $k=0,...,n$,
then we copy $F\langle p,i\rangle(n)+1$ to the output.
We ensure that the output comes into an even position if and only if $n$ is even.
As long as we do not find such a number $F\langle p,i\rangle(k)$, we just copy $0$ to the output.
It follows from the construction that:
\begin{enumerate}
\item[(a)] $F_0:\IN^\IN\to\IN^\IN$ is total and computable,
\item[(b)] $\pi_jF_0\langle p,i\rangle-1=\pi_jF\langle p,i\rangle$, if $\varphi_i$ is total,
\end{enumerate}
for all $p\in\IN^\IN$, $j\in\{1,2\}$ and $i\in\IN$.
Upon input of computable $\langle p,q\rangle$
we can now consider
\[i\in\G\langle q,\pi_2F_0\langle p,0\rangle,\pi_2F_0\langle p,1\rangle,\pi_2F_0\langle p,2\rangle,...\rangle\]
and from $i$ we can compute $r(i)\in\G(q)$, $s(i)\in\G\pi_2F_0\langle p,r(i)\rangle$ and $t(i)$ with $\varphi_{t(i)}=\varphi_{s(i)}-1=\pi_2F\langle p,r(i)\rangle$.
We can also computable $p_0:=\pi_1F_0\langle p,r(i)\rangle-1=\pi_1F\langle p,r(i)\rangle$. Then 
\[\langle p_0,t(i)\rangle\in(\langle\id\times\G\rangle\circ F\circ\langle\id\times \G\rangle)|_\mathrm{c}\langle p,q\rangle.\]
This describes the desired reduction.

(3) Given any finite number of computable $p_1,...,p_n\in\IN^\IN$, 
we can compute G\"odel numbers of the individual $p_1,...,p_n$ from any of the G\"odel
number in $\G(\langle n,\langle p_1,...,p_n\rangle\rangle)$ uniformly in $n$. 
This is a consequence of the utm- and the smn-theorem.
\qed
\end{proof}

We note that the algorithm above does not directly show $\G\star\G\equivW\G$, as for non-computable $p$ the
value $\pi_2F_0\langle p,i\rangle$ for total $\varphi_i$ might be non-computable, even though $\pi_2F\langle p,i\rangle$ is necessarily computable.
This is because the pattern of dummy symbols zero within $\pi_2F_0\langle p,i\rangle$  might depend on the non-computable $p$.
This is the price for converting $F$ into a total $F_0$. This raises the following open problem.

\begin{problem}
Is $\G\star\G\equivW\G$?
\end{problem}

Of course, $\G\star\G\equiv_\mathrm{W}^*\G$ holds, but the question is whether this is computably true.

A surprising consequence of the parallelization properties of the G\"odel problem $\G$ is that $\G$
is a natural example for a problem that is effectively discontinuous, but not computably so.
Such examples where constructed in \cite[Proposition~19]{Bra22}, but no natural example was known so far.
Intuitively, a problem $f$ is {\em effectively discontinuous}, if given a continuous function $F:\In\IN^\IN\to\IN^\IN$
that is a potential realizer of $f$, one can continuously determine some input on which $F$ fails to realize $f$.
Analogously, in the case of {\em computable discontinuity} the input is found in a computable way.
See \cite{Bra22} for more precise definition. The {\em discontinuity problem}
\[\DIS:\IN^\IN\mto\IN^\IN,p\mapsto\{q\in\IN^\IN:\U(p)\not=q\}\]
has been defined in \cite{Bra22}, where $\U:\In\IN^\IN\to\IN^\IN$ is a standard computable universal function.
We are not using the discontinuity problem in any technical way here, hence the definition
is not relevant for the following. We have $\DIS\lSW\LPO\lSW\C_\IN$ and 
in \cite[Theorem~17]{Bra22} we have proved the following for every problem $f$.

\begin{fact}[Effective discontinuity]
\label{fact:effective-discontinuity}
We obtain:
\begin{enumerate}
\item $f$ is effectively discontinuous $\iff \DIS\leq_{\mathrm{sW}}^*f$,
\item $f$ is computably discontinuous $\iff \DIS\leq_{\mathrm{sW}}f$.
\end{enumerate}
\end{fact}

In~\cite[Theorem~5.4]{Bra21} we have proved that $\widehat{\DIS}\equivW\NON$,
where 
\[\NON:\IN^\IN\mto\IN^\IN,p\mapsto\{q\in\IN^\IN:q\nleqT p\}.\]

Hence, we obtain the following corollary.

\begin{corollary}[Effective discontinuity]
\label{cor:effective-discontinuity}
$\DIS\leq_{\mathrm{sW}}^*\Kol_\geq$, but $\DIS\nleqW\G$. 
That is, $\G$ and $\Kol_\geq$ are effectively discontinuous, but not computably so.
\end{corollary}
\begin{proof}
Theorem~\ref{thm:topological} implies $\DIS\leq_\mathrm{sW}\C_\IN\equiv_\mathrm{sW}^*\Kol_\geq$.
Suppose $\DIS\leqW\G$. Then $\NON\equivW\widehat{\DIS}\leqW\widehat{\G}$ follows.
But this would imply $\NON|_{\mathrm c}\leqW\widehat{\G}|_\mathrm{c}\leqW\G$ by Proposition~\ref{thm:closure},
which is obviously wrong as $\G$ has no non-computable outputs.
\qed
\end{proof}

In some well-defined sense $\DIS$ is the weakest natural discontinuous problem~\cite{Bra21}.
That $\DIS$ is not computably reducible to $\G$ means that the uniform computational power of $\G$
is extremely weak. No other standard lower bounds of $\G$ (other than related lower bounds such as $\Kol_\geq,\G_\geq$) 
are known, besides the trivial lower bound $\id$.
The situation is quite different for the Kolmogorov complexity $\Kol$, where we have the following easy result.

\begin{proposition}[Lower bounds for Kolmogorov complexity]
\label{prop:LPO-Kolmogorov}
$\LPO\lW\Kol$ and $\lim\lW\widehat{\Kol}$.
\end{proposition}
\begin{proof}
It is easy to see that $\LPO|_\mathrm{c}\equivSW\LPO$: in order to show $\LPO\leqSW\LPO|_\mathrm{c}$ one
just has to translate the zero sequence into the zero sequence and any sequence with prefix $0^n1$
into $0^n\widehat{1}$.
Hence it suffices to show $\LPO|_\mathrm{c}\leqSW\Kol$ for the positive part of the first reduction.
There is one specific smallest G\"odel number $i$ of the constant zero sequence in $\IN^\IN$.
Given $p\in\IN^\IN$ we can just consider $\Kol(p)$ and we obtain 
\[p=\widehat{0}\iff\Kol(p)=i.\]
Hence, $\LPO|_\mathrm{c}\leqSW\Kol$ follows, which implies $\lim\equivSW\widehat{\LPO}\leqSW\widehat{\Kol}$.
However, Weih\-rauch equivalence is impossible in both statements. 
In the former case $\LPO\equivW\Kol$ contradicts $\Kol\equiv_\mathrm{W}^*\C_\IN$, which holds according to 
Theorem~\ref{thm:topological}.
In the latter case $\lim\equivW\widehat{\Kol}$ cannot hold, as $\lim\leqW\WKL'$, but $\Kol\nleqW\WKL'$
according to Corollary~\ref{cor:WKL-lim}. 
\qed
\end{proof}

As $\DIS\leqW\LPO$, we have clearly separated the G\"odel problem $\G$ from the Kolmogorov complexity $\K$.
The G\"odel problem $\G$ is also separated from its bounded counterpart $\G_\geq$ as the latter is continuous and the former not.

\begin{corollary}[G\"odel and Kolmogorov]
\label{cor:GK}
$\G_\geq\lW\G\lW\K$.
\end{corollary}

\section{Conclusions}
\label{sec:conclusions}

We have seen that the G\"odel problem and its variants have natural classifications
in the topological Weihrauch lattice. The halting problem turned out to be the optimal oracle
that validates these classifications. With respect to the computability-theoretic version of Weihrauch
reducibility the situation was more complex. 
We established an optimal upper bound with respect to the Weihrauch version of the Kirby-Paris hierarchy
 that we have used to classify the complexity. 
We have also discussed closure properties and lower bounds. 
 
Now we still want to argue why the closed and compact choice problems
form natural degrees for these purposes and why they are legitimate representatives
of the Kirby-Paris hierarchy.
We formulate a fairly general and relatively vague thesis that argues
for a necessary condition that Weihrauch degrees should satisfy in order to be legitimate counterparts 
of axiom systems in reverse mathematics.

\begin{thesis}
\label{thesis}
A necessary condition that a Weihrauch degree $d$ has to satisfy in order to be considered as a correspondent of an axiom 
system $A$ of reverse mathematics is that $d$ is Weihrauch equivalent to a suitable interpretation of $A$ as a problem
and that $d$ is closed under composition, i.e., $d\star d\equivW d$.
\end{thesis}

The reason for requesting closure under composition is that this corresponds to closure under deduction,
which is for free in reverse mathematics. We claim that the problems shown in table~\ref{fig:axiom-systems}
satisfy these requirements

\begin{figure}[htb]
\begin{center}
\begin{tabular}{cc}
\hline\rule[-1.5mm]{0pt}{5.5mm}
{\bf\color{black}Weihrauch complexity\ } & {\ \bf\color{black}Reverse mathematics}\\\hline
&\\[-0.3cm]
$\C_{\IN^\IN}$ & $\ATR_0$\\
$\lim^\diamond$ & $\ACA_0$\\
$\WKL$ & $\WKL_0^*$\\
$\C_\IN^{(n)}$ & $\I\mathrm{\Sigma}^0_{n+1}$\\
$\K_\IN^{(n)}$ & $\B\mathrm{\Sigma}^0_{n+1}$\\
$\id$    & $\RCA_0^*$\\\hline
\end{tabular}
\end{center}
\caption{Weihrauch degrees that correspond to reverse mathematics axiom systems.}
\label{fig:axiom-systems}
\end{figure}

The observation by the author and Rakotoniaina that closed and compact choice problems correspond to the induction and boundedness principles~\cite{BR17}
is essentially based on the equivalences stated in Proposition~\ref{prop:lim-inf}.
Alternative candidates for this correspondents to the induction principles have, for instance, been provided by
Davis, Hirschfeldt, Hirst, Pardo, Pauly and Yokoyama in \cite[Theorems~6--9]{DHH+20}.
They studied
two combinatorial principles $\ERT$ and $\ECT$, and they proved that $\ERT$ is equivalent to $\I\mathrm{\Sigma}^0_1$
and $\ECT$ is equivalent to $\I\mathrm{\Sigma}^0_2$ over $\RCA_0$, whereas
$\ERT\equivW\LPO^*$ and $\ECT\equivW\T\C_\IN^*$.
However if, in line with Thesis~\ref{thesis}, we take closure under composition, then we obtain:

\begin{proposition}
\label{prop:ERT-ECT}
$\ERT^\diamond\equivW(\LPO^*)^\diamond\equivW\C_\IN\mbox{ and }\ECT^\diamond\equivW(\T\C_\IN^*)^\diamond\equivSW\C_\IN'$.
\end{proposition}

Hence, we arrive back at the classes that we consider as correspondents of $\I\mathrm{\Sigma}^0_1$ and $\I\mathrm{\Sigma}^0_2$,
respectively. 
Here the first equivalence $(\LPO^*)^\diamond\equivW\C_\IN$ follows from Theorem~\ref{thm:first-order}.
The equivalence $(\T\C_\IN^*)^\diamond\equivSW\C_\IN'$ is a consequence of the latter theorem too, 
when one additionally considers the following lemma.

\begin{lemma}
\label{lem:LPO-TCN-CN}
$\LPO'\leqW\T\C_\IN\star\T\C_\IN\leqW\C_\IN'$.
\end{lemma}
\begin{proof}
The second reduction is a consequence of \cite[Corollary~8.9]{BG21a} and Corollary~\ref{cor:closure-composition}.
The first reduction can be seen as follows.
We can see $\LPO'$ as the task to decide whether the input $p\in\IN^\IN$ contains infinitely many zeros~\cite[Lemma~4.23]{BG21a}.
If we find a zero in a prefix of length $n$ of $p$, then we produce the output $A_n=\{n+1,n+2,...\}$.
If we find the next zero in a prefix of length $m>n$, then we modify the output to $A_m$.
Hence the final output $A$ is either empty (which means that the input contains infinitely many zeros)
or $A=A_n$ for some $n$ and hence $i\in A$ implies $i\geq n+1$, which means that no further
zeros appear in $p$ after position $n$. 
In any case, the output of $\T\C_\IN(A)$ is some number $i$.
Since $\LPO\leqW\C_\IN\leqW\T\C_\IN$,
we can use the second copy of $\T\C_\IN$ to decide whether the sequence $p(i+1),p(i+2),p(i+3),...$
contains a zero or not. If it contains a zero, then $p$ contains infinitely many zeros, if not, then
$p$ contains only finitely many zeros. This proves $\LPO'\leqW\T\C_\IN\star\T\C_\IN$.
\qed
\end{proof}

As $(\LPO')^\diamond\equivW\C_\IN'$ by Theorem~\ref{thm:first-order}, we obtain
the following corollary.

\begin{corollary} $\T\C_\IN^\diamond\equivW\C_\IN'$.
\end{corollary}

All this confirms that $\C_\IN^{(n)}$ and likewise $\K_\IN^{(n)}$ are legitimate correspondents of $\I\mathrm{\Sigma}^0_{n+1}$ and $\B\mathrm{\Sigma}^0_{n+1}$, respectively. 
Hence, our classification of the G\"odel problem and its variants has been performed with respect to an appropriate benchmark scale.

\bibliographystyle{splncs03}
\bibliography{C:/Users/\input{"|kpsewhich -var-value USERNAME"}/Documents/Spaces/Research/Bibliography/lit}

\end{document}